\documentclass{amsart}
\usepackage{amsfonts,amssymb,amscd,amsmath,enumerate,verbatim,newlfont,calc}
\usepackage{graphicx}
\RequirePackage[all]{xy}

\newtheorem{theorem}{Theorem}[section]
\newtheorem{theorem A}{Theorem A}
\newtheorem{theorem B}{Theorem B}
\newtheorem{lemma}[theorem]{Lemma}

\newtheorem{definition}[theorem]{Definition}
\newtheorem{corollary}[theorem]{Corollary}
\newtheorem{remark}[theorem]{Remark}

\newtheorem{proposition}[theorem]{Proposition}
\begin{document} 
\address{Department of Mathematics, Ferdowsi University of Mashhad, Mashhad, Iran}
\title{Subgroup Theorems for the $\tilde{B_0}$-invariant of groups}
\author[Z. Araghi Rostami]{Zeinab Araghi Rostami}
\author[M. Parvizi]{Mohsen Parvizi}
\author[P. Niroomand]{Peyman Niroomand}
\address{Department of Pure Mathematics\\
Ferdowsi University of Mashhad, Mashhad, Iran}
\email{araghirostami@gmail.com, zeinabaraghirostami@alumni.um.ac.ir}
\address{Department of Pure Mathematics\\
Ferdowsi University of Mashhad, Mashhad, Iran}
\email{parvizi@um.ac.ir}
\address{Department of Mathematics, Damghan University, Damghan, Iran}
\email{niroomand@du.ac.ir}
\maketitle
\begin{abstract}
U. Jezernik and P. Moravec have shown that if $G$ is a finite group with a subgroup $H$ of index $n$, then nth power of the Bogomolov multiplier of $G$, $\tilde{B_0}(G)^n$ is isomorphic to a subgroup of $\tilde{B_0}(H)$. In this paper we want to prove a similar result for the center by center by $w$ variety of groups, where $w$ is any outer commutator word.
\end{abstract}

\section{\bf{Introduction and Preliminaries}}
Noether’s problem \cite{13} is one of the fundamental problems of invariant theory, and asks as to whether the field of $G$-invariant functions ${\Bbb{C}(V)}^{G}$ is purely transcendental over $\Bbb{C}$, where $G$ is a given finite group? Artin and Mumford in \cite{2} introduced an obstruction ${H_{nr}^{2}}({\Bbb{C}(V)}^{G},{\Bbb{Q}}/{\Bbb{Z}})$ to this problem, called the unramified Brauer group of the field extension ${\Bbb{C}(V)}^{G}/{\Bbb{C}}$. Bogomolov in \cite{4} proved that unramified cohomology group ${H_{nr}^{2}}({\Bbb{C}(V)}^{G},{\Bbb{Q}}/{\Bbb{Z}})$ is canonically isomorphic to
$B_0(G)=\bigcap_{A\leq G}\ \ker {{res}_{A}^{G}}$,
where $A$ is an abelian subgroup of $G$, and ${{res}_{A}^{G}}: H^2(G,{\Bbb{Q}}/{\Bbb{Z}})\longrightarrow H^2(A,{\Bbb{Q}}/{\Bbb{Z}})$ is the usual cohomological restriction map. $B_0(G)$ is defined as the subgroup of the Schur multiplier $\mathcal{M}(G)=H^2(G,{\Bbb{Q}}/{\Bbb{Z}})$ of $G$. Kunyavskii in \cite{8} named $B_0(G)$ the \emph{Bogomolov multiplier} of $G$. Moravec in \cite{11} showed that in the class of finite groups, the group $\tilde{B_0}(G)=\mathcal{M}(G)/\mathcal{M}_0(G)$ is noncanonically isomorphic to $B_0(G)$, where the Schur multiplier $\mathcal{M}(G)$ is isomorphic to $\ker (G\wedge G \rightarrow [G,G])$ given by $x\wedge y \rightarrow [x,y]$ and $\mathcal{M}_0(G)$ is a subgroup of $\mathcal{M}(G)$ defined as $\mathcal{M}_0(G)=<x\wedge y \ | \ [x,y]=0 , \ x,y\in G>$. Moravec in \cite{11} showed that if ${F}/{R}$ be a free presentation of the group $G$, then $\tilde{B_0}(G)\cong \frac{F'\cap R}{<K(F)\cap R>}$, and it is an invariant of $G$. Recently in \cite{1}, we generalized the concept of the Bogomolov multiplier with respect to a variety of groups. In this paper we give some new results on this topic that correspond to those for the Bogomolov multiplier obtained by Jezernik and Moravec in \cite{7} and Bogomolov in \cite{5}. Here, we give some preliminaries which are needed in the next section.
\\
Let $x$ and $y$ be two elements of a group $G$. Then $[x, y]$, the commutator of $x$ and $y$, and $x^y$ denote the elements $x^{-1}y^{-1}xy$ and $y^{-1}xy$, respectively.
The left normed commutators of higher weights are defined inductively as
$[x_1 , x_2 , . . . , x_n]=[[x_1,...,x_{n-1}] , x_n]$.
If $M$ and $N$ are two subgroups of a group $G$ then $[M,N]$ denotes the subgroup of $G$ generated by all the commutators $[m,n]$ with $m\in M$ and $n\in N$. In particular, if $M=N=G$ then $[G,G]$, which is denoted by $G'$, is the derived subgroup of $G$. The lower and upper central series are denoted by ${\gamma}_{c}(G)$, where $c\geqslant 1$ and $Z_{d}(G)$ where $d\geqslant 0$, respectively.
Let $F_{\infty}$ be the free group freely generated by an infinite countable set $\{x_1, x_2,...\}$. If $u=u(x_1,...,x_s)$ and $v=v(x_1,...,x_t)$ are two words in
$F_{\infty}$, then the composite of $u$ and $v$, $u{\circ}v$ is defined as 
$$u{\circ}v=u(v(x_1,...,x_t),...,v(x_{(s-1)t+1},...,x_{st})).$$
In particular, the composite of some nilpotent words is called a
poly nilpotent word, i.e.,
${\gamma}_{{c_1+1},...,{c_t+1}}={\gamma}_{c_1+1}{\circ}{\gamma}_{c_2+1}{\circ}...{\circ}{\gamma}_{c_t+1}$, where ${\gamma}_{c_i+1} (1\leq i \leq t)$ is a nilpotent word in distinct variables. Outer commutator words are defined inductively, as follows. The word $x_i$ is an outer commutator word (or o.c word) of weight one. If $u=u(x_1,...,x_s)$ and $v=v(x_{s+1},...,x_{s+t})$ are o.c words of weights $s$ and $t$, respectively, then $w(x_1,...,x_{s+t})=[u(x_1,...,x_s),v(x_{s+1},...,x_{s+t})]$ is an o.c word of weight $s+t$.
\\
Let $V$ be a subset of $F_{\infty}$ and $\mathcal{V}$ be the variety of groups defined by the set of laws $V$. P. Hall in \cite{6} introduced the notion of the verbal subgroup $V(G)$ and the marginal subgroup $V^{*}(G)$, associated
with the variety $\mathcal{V}$ and the group $G$ (see \cite{9,12}, for more information).
\\
Let $G$ be any group with a free presentation $G\cong {F}/{R}$ and $V$ be a set of o. c. words. Then the $\tilde{B_0}$-invariant of $G$, $\mathcal{V}\tilde{B_0}(G)$ with respect to the variety $\mathcal{V}$ is defined by
$$\mathcal{V}\tilde{B_0}(G)=\frac{R\cap V(F)}{<T(F)\cap R>},$$
where $T(F)=\{v(f_1,...,f_s) \ | \ v\in V, f_i\in F, 1\leq i\leq s\}$.
\\
In \cite{1}, we showed that the $\tilde{B_0}$-invariant of the group $G$ is always abelian and independent of the chosen free presentation of $G$. As a special case, if $\mathcal{A}$ is the variety of abelian groups, then $A(F)=F'$ and $T(F)=\{[x,y] \ | \ x,y\in F\}$. So, in the finite case, the $\tilde{B_0}$-invariant of $G$ is exactly the Bogomolov multiplier
$$\tilde{B_0}(G)\cong \frac{R\cap F'}{<T(F)\cap R>}$$
If $\mathcal{N}_{c}$ is the variety of nilpotent groups of class at most $c\geq 1$, then $N_c(F)={\gamma}_{c+1}(F)$ and $T(F)=\{[x_1,...,x_{c+1}] \ | \ x_i\in F, 1\leq i\leq c+1\}$. In this case, this multiplier is called the $c$-nilpotent $\tilde{B_0}$-multiplier of $G$, and defined as follows
$$\mathcal{N}_{c}\tilde{B_0}(G)={\tilde{B_0}}^{(c)}(G)=\frac{R\cap {\gamma}_{c+1}(F)}{<T(F)\cap R>}.$$

\begin{proposition}[\cite{1}-Proposition 3.6]\label{p2.1}
Let $\mathcal{V}$ be a variety of groups defined by the set of laws $V$ and let $G_1$ and $G_2$ be groups. Then
$$\mathcal{V}{\tilde{B_0}}(G_1\times G_2) \cong \mathcal{V}{\tilde{B_0}}(G_1) \times \mathcal{V}{\tilde{B_0}}(G_2).$$
\end{proposition}

\begin{definition}[\cite{1}-Definition 4.1]\label{d2.1}
Let $\mathcal{V}$ be a variety of groups and $Q$ be a group and $N$ be a $Q$-module. The extension $1\rightarrow N \xrightarrow{\chi} G \xrightarrow{\pi} Q \rightarrow 1$ of $N$ by $Q$ is a VP extension if trivial words of elements of $Q$ have trivial lifts in $G$.
\end{definition}

\begin{definition}[\cite{1}-Definition 4.4]\label{d'2.1}
The VP extension $1\rightarrow N \xrightarrow{\chi} G \xrightarrow{\pi} Q \rightarrow 1$ of $N$ by $Q$ is marginal, if $\chi(N)\subseteq V^{*}(G)$.
\end{definition}

\begin{proposition}[\cite{1}-Proposition 4.5]\label{p2.1}
Let $1\rightarrow N \xrightarrow{\chi} G \xrightarrow{\pi} Q \rightarrow 1$ be a marginal extension. This sequence is a VP extension if and only if $\chi(N)\cap T(G)=1$, where $T(G)=\{v(f_1,...,f_s) \ | \ f_i\in F , v\in V\}$.
\end{proposition}

\begin{definition}[\cite{1}-Definition 4.6]\label{d2.2}
The normal subgroup $N$ of a group $G$ is a VP subgroup of $G$, if the extension $1 \xrightarrow{} N \xrightarrow{} G \xrightarrow{} G/N \xrightarrow{} 1$ is a VP extension.
\end{definition}
Note that, when $N$ is a marginal subgroup of $G$, $N$ is a VP subgroup if and only if $N\cap T(G)=1$.

\begin{definition}
Let $1\rightarrow N \xrightarrow{\chi} G \xrightarrow{\pi} Q \rightarrow 1$ be a marginal VP extension. This sequence is termed $\mathcal{V}$-stem, whenever $\chi(N)\subseteq V(G)$.
\end{definition}

\begin{definition}
Every $\mathcal{V}$-stem marginal VP extension $1\rightarrow N \xrightarrow{\chi} G \xrightarrow{\pi} Q \rightarrow 1$ with $N\cong \mathcal{V}{\tilde{B_0}}(G)$ is called VP cover and in this case $G$ is said to be a VP covering group of $Q$.
\end{definition}

\begin{definition}[\cite{10}-Definition 2.1]\label{d2.3}
Let $\mathcal{V}$ be a variety of groups defined by the set of laws $V$. Then $\mathcal{V}$ is said to be a Schur-Baer-variety, whenever the marginal factor $\frac{G}{V^*(G)}$ is finite of order $m$, say, then the verbal subgroup $V(G)$ is also finite and $|V(G)| \ | \ m^k$, $k\in \Bbb{N}$, for arbitrary group $G$.
\end{definition}

I. Schur in \cite{14} proved that the variety of abelian groups has the Schur-Baer property. Also, R. Baer in \cite{3} showed that the variety defined by o.c. words has the same property.

\begin{theorem}[\cite{9}-Theorem 1.17]\label{l2.4}
Let $\mathcal{V}$ be a variety defined by the set of laws $V$. Then the following conditions are equivalent
\begin{enumerate}
\item{$\mathcal{V}$ is a Schur-Baer variety,}
\item{for any finite group $G$, the Baer invariant $\mathcal{VM}(G)$ is of order dividing a power of $|G|$.}
\end{enumerate}
\end{theorem}

\begin{remark}\label{r2.5}
We know that the $\tilde{B_0}$-invariant $\mathcal{V}{\tilde{B_0}}(G)$ is a subgroup of the Baer invariant $\mathcal{VM}(G)$, so if $\mathcal{V}$ be a Schur-Baer variety, then for any finite group $G$, the $\tilde{B_0}$-invariant $\mathcal{V}{\tilde{B_0}}(G)$ is of order dividing a power of $|G|$.
\end{remark}

\section{\bf{The Main Results}}
In $2015$ U. Jezernik and P. Moravec in \cite{7}, showed that if $G$ is a finite group and $H$ is a subgroup of $G$ of index $n$, then ${{\tilde{B_0}}(G)}^n$ (the nth power of Bogomolov multiplier), is isomorphic to a subgroup of ${\tilde{B_0}}(H)$. In this section we want to generalize above result to the center by center by $w$ variety of groups, where $w$ is any outer commutator word.

\begin{lemma}\label{l3.1}
Let $A$ be a marginal VP subgroup of $H$ i.e. $A\leq V^*(H)$ and $A\cap T(H)=1$. If $G\cong {H}/{A}$, then $V(H)\cap A$ is a homomorphic image of $\mathcal{V}{\tilde{B_0}}(G)$.
\end{lemma}
\begin{proof}
Let $H\cong {F}/{R}$ be a free presentation of $H$. Then for some subgroup $K$ of $F$, we have $A\cong {K}/{R}$. Hence $G\cong {F}/{K}$ and by using the definition
$$\mathcal{V}{\tilde{B_0}}(G)=\frac{K\cap V(F)}{<T(F)\cap K>} \ \ ; \ \ T(F)=\{v(f_1,...,f_s) \ | \ v\in V , f_i\in F ; 1\leq i\leq s\}$$
and Dedekind's law, we have
\begin{align*}
V(H)\cap A &=V(\frac{F}{R})\cap \frac{K}{R}=\frac{V(F)R}{R}\cap \frac{K}{R}=\frac{V(F)R\cap K}{R}\\
&=\frac{(V(F)\cap K)R}{R}\cong \frac{V(F)\cap K}{(V(F)\cap K)\cap R}=\frac{V(F)\cap K}{V(F)\cap R}.
\end{align*}

Since ${K}/{R}$ is a marginal VP subgroup of ${F}/{R}$, by using \ref{d2.2}, $({K}/{R})\cap T({F}/{R})=1$. So $<T(F)\cap K>\subseteq R$. Thus we have
\begin{align*}
V(H)\cap A =\frac{V(F)\cap K}{V(F)\cap R}\cong \frac{\frac{V(F)\cap K }{<T(F)\cap K >}}{\frac{V(F)\cap R }{<T(F)\cap K>}}=\frac{\mathcal{V}{\tilde{B_0}}(G)}{\frac{V(F)\cap R}{<T(F)\cap K>}}.
\end{align*}
Therefore the result holds.
\end{proof}
\begin{corollary}\label{c3.2}
Let $\mathcal{V}$ be a Schur-Baer variety and $G$ be a finite group. If $H$ be a group with a marginal VP subgroup $A$, such that ${H}/{A}\cong G$. Then $V(H)\cap A$ is isomorphic to a subgroup of $\mathcal{V}{\tilde{B_0}}(G)$.
\end{corollary}
\begin{proof}
Using Lemma \ref{l3.1}, $V(H)\cap A$ is a homomorphic image of $\tilde{B_0}$-invariant $\mathcal{V}{\tilde{B_0}}(G)$. Also, Remark \ref{r2.5} and Definition $3.1$ in \cite{1} show that $\mathcal{V}{\tilde{B_0}}(G)$ is a finite abelian group. Hence according to the property of finite abelian groups, $V(H)\cap A$ is isomorphic to a subgroup of $\mathcal{V}{\tilde{B_0}}(G)$.
\end{proof}
Let $\mathcal{W}$ be a variety defined by an outer commutator word $w=w(x_1,...,x_r)$ and $\mathcal{V}$ be the center by center by $\mathcal{W}$ variety of groups. We can see that the variety $\mathcal{V}$ is defined by the commutator word $[w,x_1,x_2]$. Here, we work with this variety.
\begin{lemma}[\cite{10}-Lemma 3.5]\label{l3.3}
Let $B$ a subgroup of a group $E$ with finite index $n$ and $\mathcal{V}$ be the some variety of groups same as before. If $x\in V^*(E)\cap V(E)$ then for all $y_1,...,y_r\in B$, $[w(y_1,...,y_r),x^n]\in V(B)$. In particular, if $L\subseteq V^*(E)\cap V(E)$, then $[W(B),L^n]\subseteq V(B)$.
\end{lemma}
Now we want to prove the main result of this section.

\begin{theorem}\label{t3.4}
Let $\mathcal{V}$ be the variety of groups is defined by the commutator word $[w,x_1,x_2]$ same as before and $G$ be a finite group with a VP cover $1\rightarrow L \rightarrow G^* \rightarrow G \rightarrow 1$. If $H$ be a subgroup of index $n$ in $G$ and for some subgroup $B$ of $G^*$, $H\cong \frac{B}{L}$, then $[W(B),\mathcal{V}{\tilde{B_0}}(G)^n]$ is isomorphic to a subgroup of $\mathcal{V}{\tilde{B_0}}(H)$.
\end{theorem}
\begin{proof}
We know that, $L\cong \mathcal{V}{\tilde{B_0}}(G)$, $L\leq V(G^*)\cap V^*(G^*)$ and $L\cap T(G^*)=1$. Since $[H:G]=n$, so $[G^*:B]=n$. Also by using Lemma \ref{l3.3}, $[W(B),L^n]\subseteq V(B)$. On the other hand, we have $L\subseteq V^*(B)$ and $L\cap T(B)=1$. So $L$ is a marginal VP subgroup in $B$. Also, since the variety $\mathcal{V}$ has the Schur-Baer property, Corollary \ref{c3.2} show that $L\cap V(B)$ is isomorphic to a subgroup of $\mathcal{V}{\tilde{B_0}}(H)$ and $[W(B),\mathcal{V}{\tilde{B_0}}(G)^n]\subseteq L\cap V(B)$. Hence $[W(B),\mathcal{V}{\tilde{B_0}}(G)^n]$ is isomorphic to a subgroup of $\mathcal{V}{\tilde{B_0}}(H)$.
\end{proof}
Now introduce the following corollaries, which generalize the Proposition $6.2$ in \cite{7} and Lemma $2.6$ in \cite{5}.
\begin{corollary}\label{c3.5}
Let $G$ be a finite group and $G_p$ be a sylow p-subgroup of $G$. If $G$ has a VP cover $1\rightarrow L\rightarrow G^* \rightarrow G \rightarrow 1$ with $G_p \cong {B}/{L}$. Then $[W(B),\mathcal{V}{\tilde{B_0}}(G)_p]$ is isomorphic to a subgroup of $\mathcal{V}{\tilde{B_0}}(G_p)$. 
\end{corollary}
\begin{proof}
Let $G$ be a finite group of order ${p^{\alpha}}n$, such that $(p,n)=1$. Then $[G,G_p]=n$ and by using Theorem \ref{l2.4}, $\mathcal{V}{\tilde{B_0}}(G)$ has an order dividing a power of $|G|={p^{\alpha}}n$. So $\mathcal{V}{\tilde{B_0}}(G)^n=\mathcal{V}{\tilde{B_0}}(G)_p$. Hence the proof is completed by using Theorem \ref{t3.4}.
\end{proof}

\begin{corollary}\label{c3.6}
Based on the assumptions and the notations of theorem \ref{t3.4}, if $\mathcal{V}{\tilde{B_0}}(G)^{[n]}$ be the set of all elements of $\mathcal{V}{\tilde{B_0}}(G)$ of order coprime to $n$. Then $[W(B),\mathcal{V}{\tilde{B_0}}(G)^{[n]}]$ is isomorphic to a subgroup of $\mathcal{V}{\tilde{B_0}}(H)$.
\end{corollary}
\begin{proof}
We know that $G^{[n]}$ is a subgroup of $G^n$. Thus the result follows by using Theorem \ref{t3.4}.
\end{proof}

\begin{corollary}\label{c3.7}
Based on the assumptions and the notations of theorem \ref{t3.4},
let $d$ and $e$ be the exponents of $[W(B),\mathcal{V}{\tilde{B_0}}(G)]$ and $\mathcal{V}{\tilde{B_0}}(H)$, respectively. Then $d$ divides $ne$. In particular, $d$ divides $n$ whenever $H$ is cyclic subgroup of $G$.
\begin{proof}
According to the assumption, $\mathcal{V}{\tilde{B_0}}(G)\cong L\leq V^*(G^*)$, $L\leq B$ and $\mathcal{V}{\tilde{B_0}}(G)\leq V^*(B)$. So the Turner-Smith's Theorem \cite{15} shows that
$$[W(B),\mathcal{V}{\tilde{B_0}}(G)^n]=[W(B),\mathcal{V}{\tilde{B_0}}(G)]^n.$$
Therefore by using Theorem \ref{t3.4}, we have
$$[W(B),\mathcal{V}{\tilde{B_0}}(G)]^n=[W(B),\mathcal{V}{\tilde{B_0}}(G)^n]\leq \mathcal{V}{\tilde{B_0}}(H)$$
Hence, $[W(B),\mathcal{V}{\tilde{B_0}}(G)]^{ne}\leq {\mathcal{V}{\tilde{B_0}}(H)}^e=1$, and so $d$ divides $ne$.
\end{proof}

\end{corollary}

In $2002$ F. Bogomolov [\cite{5}-Lemma 2.6] proved that if $G$ is a finite group and $G_p$ is a sylow p-subgroup of $G$, then the sylow $p$-subgroup of the Bogomolov multiplier of $G$ $(i.e.\ \tilde{B_0}(G)_p)$ is isomorphic to a subgroup of the Bogomolov multiplier of its sylow $p$-subgroup $\tilde{B_0}(G_p)$. Now we want to generalize this problem for an arbitrary variety.

\begin{lemma}\label{l3.8}
Let $G={H_1}\times {H_2}$ be the direct product of two finite $p$ and $q$-groups $H_1$ and $H_2$, respectively. Then $\mathcal{V}\tilde{B_0}(G)_p\cong \mathcal{V}\tilde{B_0}(G_p)$.
\end{lemma}
\begin{proof}
We know $\mathcal{V}{\tilde{B_0}}(G_p)=\mathcal{V}{\tilde{B_0}}(H_1)$. By using Proposition \ref{p2.1}, we have
$$\mathcal{V}{\tilde{B_0}}({H_1}\times {H_2})\cong \mathcal{V}{\tilde{B_0}}(H_1)\oplus \mathcal{V}{\tilde{B_0}}(H_2).$$
Now by Remark \ref{r2.5}, $\mathcal{V}{\tilde{B_0}}(H_i)$ is a $p$-group and $q$-group for $i=1,2$, respectively. Hence $\mathcal{V}{\tilde{B_0}}(G)_p\cong \mathcal{V}{\tilde{B_0}}(H_1)=\mathcal{V}{\tilde{B_0}}(G_p)$.
\end{proof}

\begin{lemma}\label{l3.9}
Let $G$ be a finite nilpotent group with a sylow $p$-subgroup $G_p$. Then $\mathcal{V}{\tilde{B_0}}(G)_p\cong \mathcal{V}{\tilde{B_0}}(G_p)$.
\end{lemma}
\begin{proof}
It is known that a finite nilpotent group is the direct product of it's sylow $p$-subgroups. Thus $G=G_{p_1}\times G_{p_2}\times ... \times G_{p_k}$. The proof is completed by using induction on $k$ and applying Lemma \ref{l3.8}.
\end{proof}

\begin{corollary}\label{c3.10}
Let $G$ be a finite nilpotent group. If every sylow $p$-subgroup of $G$ has a trivial $\tilde{B_0}$-invariant, then so has $G$.
\end{corollary}
\begin{proof}
Using Remark \ref{r2.5}, $\mathcal{V}{\tilde{B_0}}(G)$ has an order dividing a power of $|G|$. Now Lemma \ref{l3.9} gives the result.
\end{proof}

\end{document}